\documentclass[12pt]{article}
\usepackage[margin=1in]{geometry}  % set the margins to 1in on all sides
\usepackage{graphicx}             % to include figures
\usepackage{amsmath}               % great math stuff
\usepackage{amsfonts}              % for blackboard bold, etc
\usepackage{amsthm}                % better theorem environments
\usepackage{graphics}              % or graphicx
\usepackage{epstopdf}
\usepackage[utf8]{inputenc}
\usepackage[english]{babel}

\newtheorem{thm}{Theorem}[section]
\newtheorem{lem}[thm]{Lemma}

\newtheorem{rem}[thm]{Remark}

\date{}
\usepackage{amssymb, amsmath}
\usepackage{multirow}
\usepackage{graphicx}
\begin{document}

\title{A new bound for the crossing number of wrapped butterflies}

\author{Vijaya N$^{1, 2}$, Bharati R$^{3}$, Paul Manuel$^{4}$\\
\small{$^{1}$Research and Development Centre, Bharathiar University, Coimbatore 621 026, India} \\
\small{$^{2}$Department of Mathematics, Alpha College of Engineering, Chennai 600 124, India} \\
\small{$^{3}$Department of Mathematics, Loyola College, Chennai 600 034, India}\\
\small{$^{4}$Department of Information Science, Kuwait University, Safat, Kuwait, 13060}}
\maketitle

\begin{abstract}
We fix an error in the bound obtained in \cite{ViBrPmIv15} for the crossing number of wrapped butterflies. The new bound finer than the one provided earlier.\\\\
{\bf AMS Subject Classification:} 05C05\\\\
{\bf Keywords:} Crossing number, good drawing of a graph, butterfly network, wrapped butterfly network.
\end{abstract}
\section{Introduction}
Let $G$ be a simple connected graph with vertex set $V(G)$ and edge set $E(G)$. A drawing of $G$ is said to be good provided that no edge crosses itself, no adjacent edges cross each other, no two edges cross more than once, and no three edges cross in a point. The crossing number $Cr(G)$ of a graph $G$ is the minimum possible number of edge crossings in a good drawing of $G$ in the plane. 

Garey and Johnson \cite{GaJo83} proved that computing the crossing number is NP-complete. Not surprisingly, there are only a few infinite families of graphs for which the exact crossing numbers are known (see for example \cite{LiYaZh09, PaRi07, RiTh95}). Therefore, it is more practical to determine the upper and lower bounds for the crossing number of a graph.

Another family of graphs whose crossing numbers have received a good deal of attention is the interconnection networks proposed for parallel computer architecture. For hypercubes and cube connected cycles, the crossing number problem is investigated by Sýkora \emph{et al}. \cite{SyVr93}. Cimikowski  \cite{Ci96} has given upper bounds for the crossing number for various networks like torus, buttery and Benes networks. He has also obtained a bound for the crossing number of mesh of trees \cite{Ci03}. Manuel \emph{et al}. \cite{PaBhInVa13} have given improved bounds for the crossing number of butterfly network and have also given a lower bound which matches the upper bound obtained.

In an earlier paper Vijaya et.al \cite{ViBrPmIv15} modified the drawing of the $r$-dimensional wrapped butterfly network $WB(r)$ with $\frac{5}{4}4^{r}-3(2^{r})-r2^{r}$ crossings which slightly improved the existing estimate $\frac{3}{2}4^{r}-3(2^{r})-r2^{r}$ given by Cimikowski \cite{Ci96}. They proposed a new drawing and claimed to have obtained a finer bound. In this paper we fix an error identified there and obtain a better bound with $ \frac{7}{8}4^{r}-(3r-4)2^{r}$ crossings. A comparison chart is also provided in the last section.

\section{Wrapped butterfly network}

The set of vertices of an $r$-dimensional butterfly corresponds to pairs $[w, i]$, where $i$ is the dimension or level of a vertex $(0 \leq i \leq r)$ and $w$ is an $r$-bit binary number that denotes the row of the vertex. Two vertices $[w, i]$ and $[w', i']$ are linked by an edge if and only if $i' = i + 1$ and either:

\begin{enumerate}
  \item $w$ and $w'$ are identical, or
  \item $w$ and $w'$ differ in precisely the $i^{th}$ bit.
\end{enumerate}

The $r$-dimensional butterfly network denoted by $BF(r)$ has $(r + 1)2^{r}$ vertices and $r2^{r + 1}$ edges. It has $r + 1$ levels and there are $2^{r}$ vertices in each level. Each vertex on level 0 and level $r$ is of degree 2; all other vertices are of degree 4 \cite{Le92}. 

When the vertices of $BF(r)$ in level 0 are merged with those in level $r$, a new structure called the wrapped butterfly is obtained. The $r$-dimensional wrapped butterfly denoted by $WB(r)$ has $r$ levels, from 0 to $r - 1$, and each level has $2^{r}$ vertices \cite{Le92}. The wrapped butterfly network has been studied with regard to Hamiltonian paths and cycles \cite{BaRa94, Wo95} and VLSI layout \cite{Ke94}.

Another level wise labeling scheme for the vertices of $WB(r)$ is adopted here. The vertices are numbered from left to right, with integers $1, 2 \cdots 2^{r}$. A vertex in level $i$ and in position $j$ from the left is designated by the pair $(i, j), \; 1\leq j \leq 2^{r}, \; 0 \leq i \leq r-1$. The labeling in $WB(3)$ is illustrated in Figure \ref{Fig1}.

\begin{figure}[h]
 \centering
  \includegraphics[width=9 cm]{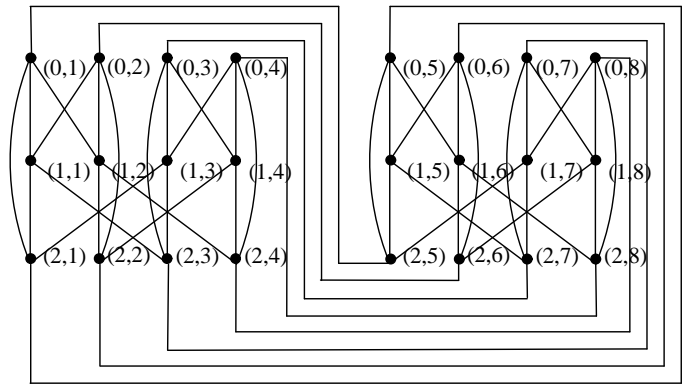}
  \vspace{0 cm}
  \caption{Another labeling of $WB(3)$}
  \label{Fig1}
\end{figure}

\section{A finer bound for the crossing number}

This section begins with the statements of existing bounds.
\begin{thm}\label{T1}
\cite{Ci96} Let $G$ be $WB(r)$. Then $Cr(G)\leq \frac{3}{2}4^{r}-3(2^{r})-r2^{r}$
\end{thm}
\begin{thm}\label{T2}
\cite{ViBrPmIv15} Let $G$ be $WB(r)$. Then $Cr(G)\leq \frac{5}{4}4^{r}-3(2^{r})-r(2^{r}) $
\end{thm}

\begin{thm}\label{T3}
\cite{ViBrPmIv15} Let $G$ be $RWB(r)$. Then $Cr(G)\leq \frac{1}{2}4^{r}-2^{r+1}$
\end{thm}

To recall the definition of the new drawing proposed in \cite{ViBrPmIv15} define cycles $C_{k, j} = (0, 4(k - 1) + j) (1, 4(k - 1) + j)\ldots (r - 1, 4(k - 1) + j) (0, 4(k - 1) + j)$, $1 \leq j \leq 4$; $1 \leq k \leq 2^{r-2}$.

Let $B_{k} , 1 \leq k\leq 2^{r-2}$,  be the subgraph of $WB(r)$ induced by the cycles $C_{k, j}, 1 \leq j \leq4$.

\begin{figure}[h]
 \centering
  \includegraphics[width=11 cm]{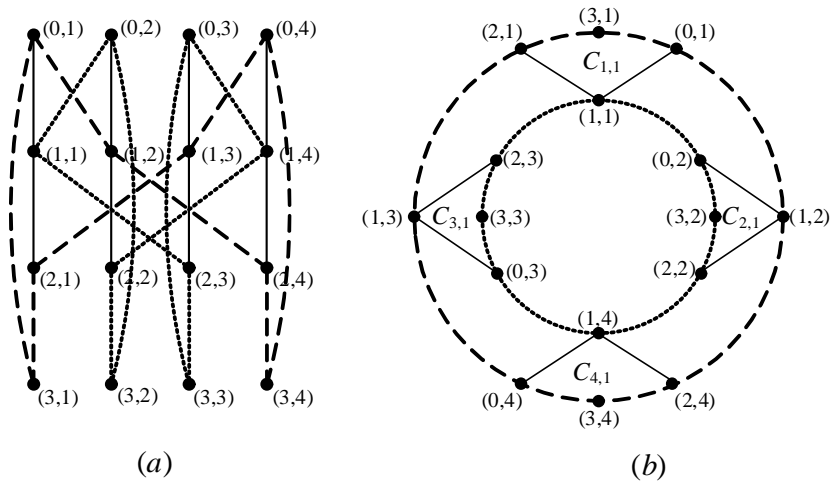}
  \vspace{0 cm}
  \caption{(a). The graph $B$ in $WB(4)$ (b). The ring $R$ in $RB(4)$}
  \label{Fig2}
\end{figure}

\begin{lem}\label{L2}
\cite{ViBrPmIv15} The graphs $B_{k}, 1 \leq k \leq 2^{r - 2}$, are isomorphic.
\end{lem}

\begin{lem}\label{L3}
\cite{ViBrPmIv15} The graphs $B_{k}, 1 \leq k \leq 2^{r - 2}$, are planar.
\end{lem}

The plane graph associated with $B_{1}$ is called a ring and is denoted by $R$; see Figure \ref{Fig2}. The proposed drawing of $WB(r)$, denoted $RB(r)$, is made up of rings $R_{k}$, $1 \leq k \leq 2^{r - 2}$. The ring $R_{j}$ is drawn in the interior face of the ring $R_{i}$ whenever $1 \leq i < j\leq 2^{r - 2}$. Since each $R_{k}$ is planar, the edges of $RB(r)$ contributing to the crossing number are the ones corresponding to the cross edges of $WB(r)$ between levels $i$ and $i + 1$, $2 \leq i \leq r - 2, r \geq 4$, and the ones corresponding to the wraparound edges of $WB(r)$. We call these edges as inner edges of $RB(r)$ and they cross the rings and cross one another too. Before attempting to count the number of crossings in $RB(r)$ we first describe the method of drawing; a few more definitions and notations are necessary.

The set of all edges belonging to the rings $R_{k}$, $1 \leq k \leq 2^{r - 2}$, is denoted by $RE$. An inner edge of $RB(r)$ corresponding to a cross edge of $WB(r)$ between levels $i$ and $i + 1$ is denoted by $I_{i}$, $2 \leq i \leq r - 2$, $r \geq 4$. The set of all such inner edges of $RB(r)$ is denoted by $IE$. An edge of $RB(r)$ corresponding to a wraparound edge of $WB(r)$ is also an inner edge and is denoted by $I_{r - 1}$. The set of all such edges is denoted by $WIE$.

It is clear from the structure of the ring $R_{1}$ that the cycles $C_{1, 1}, C_{1, 2}, C_{1, 4}$ and $C_{1, 3}$ appear clockwise in this order. The same is true for any ring $R_{k}$. Thus the cycles $C_{1, 1}, C_{2, 1}, C_{3,1} \ldots C_{2^{r - 2}, 1}$ appear one below the other in $RB(r)$. This observation is useful in describing the method of including the edges of $RB(r)$. The graph induced by the vertices  in $\bigcup_{k=1}^{2^{r-2}}V(C_{k,1})$  constitutes one fourth of the new drawing. If $D$ is a drawing of $RB(r)$ then the one fourth of the diagram is denoted, for convenience, by $\frac{1}{4} D$.

In $\frac{1}{4} D$ the inner edges $I_{2}$ join the vertices of $C_{1, 1}$ and $C_{2, 1}$; $C_{3, 1}$ and $C_{4, 1}$ and so on. All these edges are drawn on left; see Figure \ref{Fig3}. Inner edges $I_{3}$ join the vertices of $C_{1, 1 }$ and $C_{3, 1}$; $C_{2, 1}$ and $C_{4, 1}$ and so on. These edges are distributed equally on the left and right. This process continues and gives the diagram $\frac{1}{4} D$. Note that in $\frac{1}{4} D$ the inner edges $I_{i}$, $2 \leq i \leq r - 1$, occur in pairs.

\begin{lem}\label{L4}
$Cr_{D}(IE, RE)=2^{r+1}(2^{r-3}-1), r\geq4$.
\end{lem}
\begin{proof}
A cross edge in $WB(r)$ between levels $i$ and $i + 1$ crosses $2^{i - 2}$ number of $B_{k}$'s and therefore the edge $I_{i}$ in $RB(r)$ crosses $2^{i - 2}$ number of $R_{k}$'s. The number of crossings of an $I_{i}$ with a ring is 2 and there are $2^{r}$ number of such edges. Hence the number of crossings of $I_{i}$ with the rings equals $2 \times 2^{i - 2} \times 2^{r}$. This is true for every $i$, $2 \leq i \leq r - 2$. This count adds up to
$\sum_{i=2}^{r-2}2^{r+i-1}=2^{r+1} (2^{r-3}-1)$.
\end{proof}

\begin{lem}\label{L5}
$Cr_{D}(WIE, RE)=2^{2r-2}, r\geq4$.
\end{lem}

\begin{proof}
Now each wraparound edge in $WB(r)$, $r \geq 4$, crosses $2^{r - 3}$ number of $B_{k}$'s. Hence an edge $I_{r-1}$ crosses $2^{r - 3}$ number of rings. The intersection of a wraparound edge and a ring accounts for 2 crossings and there are $2^{r}$ number of wraparound edges.  Hence the number of crossings of wraparound edges and the rings equals $2\times 2^{r - 3}\times2^{r} = 2^{2r - 2}$.
\begin{figure}[h]
 \centering
  \includegraphics[width=16 cm]{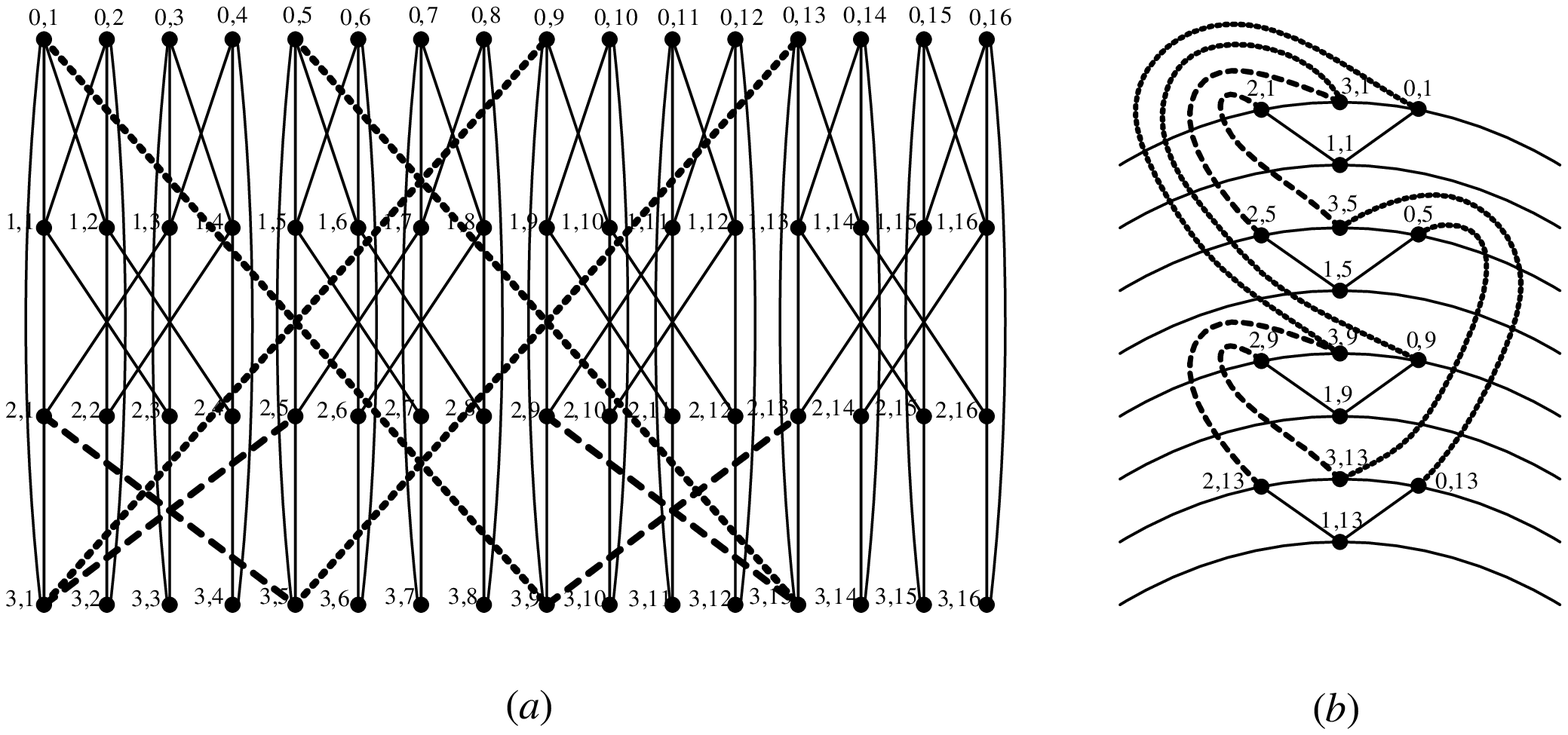}
  \vspace{0 cm}
  \caption{One fourth diagram of $RB(4)$}
  \label{Fig3}
\end{figure}
\end{proof}

\begin{lem}\label{L6}
$Cr_{D}(IE, WIE)=2^{r+1}(2^{r-4}-1), r\geq4$.
\end{lem}

\begin{proof}
A close analysis of $ \frac{1}{4}D$, where $D$ is a good drawing of $RB(r)$ gives the following information: an edge $I_{3}$ crosses 2 edges in $WIE$; an $I_{4}$ crosses 4 edges in $WIE$ and so on. Thus $I_{r - 2}$ crosses $2^{r - 4}$ edges in $WIE$. There are $2^{r - 2}$ number of edges each of type $I_{i}$, $3 \leq i \leq r - 2$. Thus the number of crossings of the inner edges with the wraparound edges in $ \frac{1}{4}D$ is given by $2^{r-2}(2+2^{2}+\ldots+2^{r-4})$. Thus $Cr_{D} (IE, WIE)=4\cdot2^{r-2} (2+2^{2}+\ldots+2^{r-4})=2^{r+1} (2^{r-4}-1)$.

\end{proof}

\begin{lem}\label{L7}
$Cr_{D}(WIE, WIE)=2^{r}(2^{r-4}-1), r\geq4$.
\end{lem}

\begin{proof}
The pair of wraparound edges incident vertices of $C_{1, 1}$ crosses wraparound edges incident at vertices of $C_{3, 1 },C_{5, 1} \ldots C_{2^{r - 3}- 1, 1}$. Each pair intersection causes 4 crossings and there are $2^{r - 4 }-1$ number of cycles. This accounts for $4(2^{r - 4 }- 1)$ crossings. Similarly the pair of wraparound edges incident vertices of  $C_{3, 1 }$ crosses wraparound edges incident at vertices of $C_{5, 1 }, C_{7, 1} \ldots C_{2^{r - 3}- 1, 1}$. This gives a count of $4(2^{r - 4} - 2)$. Proceeding in this manner the count adds up to
$4(2^{r - 4} - 1) + 4(2^{r - 4} - 2) + \ldots +4(2^{r - 4 }- (2^{r - 4} - 1)) = 2^{r - 3} (2^{r - 4} - 1)$
Similarly the edges incident at cycles $C_{2, 1}, C_{4, 1} \ldots C_{2^{r - 3} - 2, 1}$ give the same count. Thus the number of crossings of the wraparound edges with themselves in $\frac{1}{4}D$ is $2\cdot 2^{r - 3} (2^{r - 4} - 1)$. So $Cr_{D} (WIE,WIE)=2^{r} (2^{r-4}-1)$.
\end{proof}

\begin{rem}\label{R1}
Lemmas \ref{L6} and \ref{L7} confirm that wraparound edges of $RB(r)$ do not cross themselves and the inner edges when $r = 4$. But they cross when $r > 4$.
\end{rem}

\begin{rem}\label{R2}
As in ordinary butterfly networks one cannot say that $r$-dimensional wrapped butterfly network $RB(r)$ contains two copies of $(r - 1)$-dimensional networks. A close examination of the diagram of $RB(r)$ implies that $Cr(RB(r))$ includes $2Cr(RB(r - 1))$ along with additional counts arising from the crossings, in $RB(r)$, of the wraparound edges with the rings, inner edges and wraparound edges.
\end{rem}
	If $D$ is good drawing of $RB(4)$ it follows from Figure \ref{Fig3} that $\frac{1}{4} D$ has 24 crossings. Thus we have the following result.

\begin{lem}\label{L9}
$Cr_{D}(RB(4))=96$.
\end{lem}
Twice the count $Cr(RB(4))$ together with Lemmas \ref{L5}, \ref{L6} and \ref{L7} for $r =5$ gives the following result.

\begin{lem}\label{L10}
$Cr_{D}(RB(5))=544$.
\end{lem}

\begin{figure}
\centering
  \includegraphics[width=11 cm]{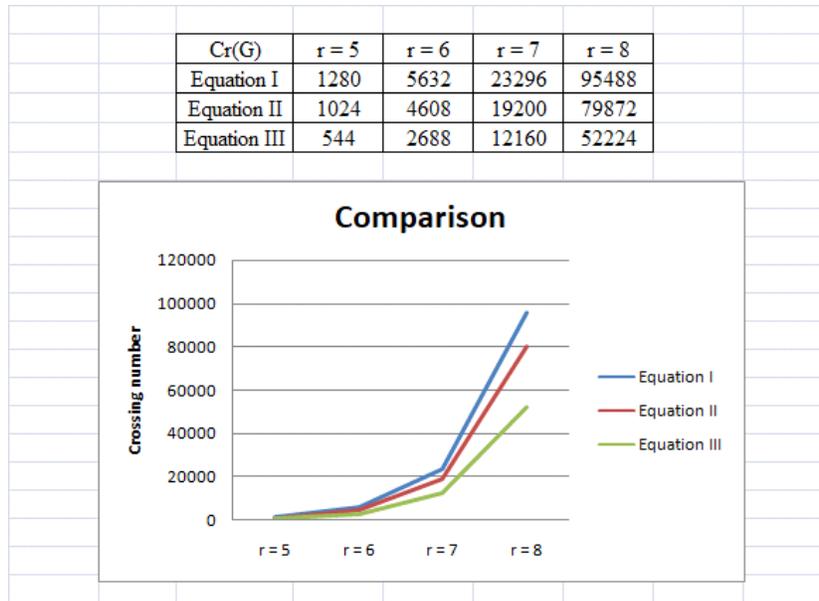}
  \vspace{0 cm}
  \caption{Comparison of the three bounds}
  \label{Fig4}
\end{figure}

\begin{thm}\label{T4}
Let $G$ be $RB(r)$. Then for $r \geq 6$,  $Cr(G)\leq \frac{7}{8}4^{r}-(3r-4)2^{r}$.
\end{thm}
\begin{proof}
The proof is by induction on $r$. The result is true for $r = 5$. Assume the result for $r = k - 1$.
Then $Cr(RB(k-1))\leq \frac{7}{8} 4^{k-1}-(3(k-1)-4) 2^{k-1}=\frac{7}{8} 4^{k-1}-(3k-7)2^{k-1}$. By an earlier remark, $Cr(RB(k))\leq 2 \{ \frac{7}{8} 4^{k-1}-(3k-7) 2^{k-1}\}+2^{2k-2}+2^{k+1}(2^{k-4}-1)+2^{k}(2^{k-4}-1)$\\
$=\frac{7}{8} 4^{k}-(3k-4)2^{k}$
\end{proof}

The bound mentioned in Theorem \ref{T3} was obtained taking into consideration the crossings of only the inner edges and wraparound edges with the rings. We claim that our new count is correct.

Denoting the bounds in Theorems \ref{T1}, \ref{T2} and \ref{T4} as equations I, II and III respectively the comparison diagram is given in Figure \ref{Fig4}.

\end{document}